\documentclass[11pt]{article}
\usepackage[T1]{fontenc}
\usepackage[utf8]{inputenc}
\usepackage{lmodern}
\usepackage[margin=1in]{geometry}
\usepackage{amsmath,amssymb,amsthm,mathtools}
\usepackage{microtype}
\usepackage[colorlinks=true,linkcolor=blue,citecolor=blue,urlcolor=blue]{hyperref}
\usepackage{enumitem}
\usepackage{needspace}
\setlist{itemsep=3pt,topsep=5pt}
\newtheorem{theorem}{Theorem}[section]
\newtheorem{proposition}[theorem]{Proposition}
\newtheorem{lemma}[theorem]{Lemma}
\newtheorem{corollary}[theorem]{Corollary}
\theoremstyle{definition}
\newtheorem{definition}[theorem]{Definition}
\newtheorem{example}[theorem]{Example}
\theoremstyle{remark}
\newtheorem{remark}[theorem]{Remark}
\newcommand{\F}{\mathcal F}
\newcommand{\HH}{\mathcal H}
\newcommand{\PP}{\mathbb P}
\newcommand{\CC}{\mathbb C}
\newcommand{\Vol}{\operatorname{Vol}}
\newcommand{\Sing}{\operatorname{Sing}}
\title{Bounded Leaf Volumes and Analytic Quotients\\of Singular Holomorphic Foliations}
\author{Carlos Eduardo Mart\'inez \and Alberto Verjovsky}
\date{}
\begin{document}
\maketitle

\begin{center}
\textit{To the memory of Dr. Rodolfo Patricio Mart\'inez y Romero
(1950--2024)}

\end{center}

\begin{abstract}
Let $\F$ be a singular holomorphic foliation of dimension $k$ on a compact K\"ahler manifold $M$ of dimension $n$. Suppose that the volumes of its regular leaves are uniformly bounded, $\dim_{\CC}\Sing\F<k$, and the union of the leaf closures meeting $\Sing\F$ has finite $2k$-dimensional Hausdorff measure. We prove that the quotient by chains of intersecting leaf closures is a normal compact complex analytic space of dimension $n-k$. The quotient map is proper and holomorphic, with connected fibers. The proof uses Bishop's theorems, Remmert--Stein extension, Barlet's cycle space, and Stein factorization. Away from the images of separatrices, the quotient has local charts given by finite holonomy groups. Examples show that these images need not coincide with the singular set of the quotient, and that bounded leaf volumes alone do not imply Hausdorffness. In codimension one, Malgrange's theorem makes the separatrix-measure assumption automatic when $\operatorname{codim}\Sing\F\ge3$.
\end{abstract}

\noindent\small\textit{2020 Mathematics Subject Classification.}
Primary 32S65; Secondary 32M25, 32C25, 32Q15.\par
\noindent\textit{Keywords.} Singular holomorphic foliations, bounded leaf volumes,
analytic quotients, K\"ahler manifolds, separatrices, holonomy, Bishop's theorem,
singular Frobenius theorem.\par\normalsize

\section{The global quotient on a compact K\"ahler manifold}

Let $M$ be a connected compact K\"ahler manifold of complex dimension $n$, with K\"ahler form $\omega$. Write $\mathcal O_M$ for its sheaf of holomorphic functions and $\mathcal T_M$ for its sheaf of holomorphic vector fields. A \emph{singular holomorphic foliation} $\F$ of rank $k$, where $0<k<n$, is a coherent subsheaf $\mathcal T_{\F}\subset\mathcal T_M$ of generic rank $k$ such that
\[
 [\mathcal T_{\F},\mathcal T_{\F}]\subset\mathcal T_{\F},
 \qquad \mathcal T_M/\mathcal T_{\F}\ \text{is torsion-free}.
\]
\emph{A coherent sheaf} is locally finitely generated, and the relations among any finite family of local generators are locally finitely generated. The involutivity/integrability condition means that $[X,Y]$ is tangent to the foliation whenever $X$ and $Y$ are local sections of $\mathcal T_{\F}$, this guarantees the existence of tangent manifolds by the holomorphic Frobenius theorem. The torsion-free condition gurantees that, for germs $f\in\mathcal O_{M,x}$ and $v\in\mathcal T_{M,x}$, $f\ne0$ and $fv\in\mathcal T_{\F,x}$ imply $v\in\mathcal T_{\F,x}$, usually a subsheaf with this property is called \emph{saturated}.

The singular set $S=\Sing(\F)$ is the analytic set where $\mathcal T_M/\mathcal T_{\F}$ is not locally free, its (real) codimension is at least two, although we will use complex dimensions and codimensions throughout the rest of the article. On $M^\circ=M\setminus S$, the tangent sheaf of the foliation is a holomorphic subbundle of rank $k$. The holomorphic Frobenius theorem gives local coordinates $(z,w)\in\CC^k\times\CC^{n-k}$ in which it is spanned by $\partial/\partial z_1,\ldots,\partial/\partial z_k$. We call the local a \emph{foliated chart} and sets $w=\mathrm{constant}$ are called \emph{plaques}. A regular leaf is a maximal connected immersed complex submanifold tangent to this subbundle; its dimension is $k$, and the codimension of the foliation is $n-k$.

The \emph{conormal sheaf} consists of holomorphic one-forms annihilating $\mathcal T_{\F}$. In codimension one it has a local generator $\alpha$ whose coefficients have no common divisorial factor. Then $S=\{\alpha=0\}$ locally, and integrability is equivalent to $\alpha\wedge d\alpha=0$. For higher-codimension Pfaff systems, see Section~\ref{sec:malgrange}.

Throughout this paper, a leaf means a regular leaf, since $M$ is K\"ahler, its volume is calculated as
\[
 \Vol_{2k}(L)=\frac1{k!}\int_L\omega^k.
\]
All closures are taken in $M$ unless another ambient space is specified. Write $A_L=\overline L$.

\begin{definition}
or $x,y\in M$, write $x\sim y$ if $x=y$, or if there are regular leaves $L_1,\ldots,L_m$ such that
\[
 x\in A_{L_1},\qquad y\in A_{L_m},\qquad
 A_{L_i}\cap A_{L_{i+1}}\ne\varnothing\quad(1\leq i<m).
\]
The space $Q=M/{\sim}$ is equipped with the quotient topology, and $q:M\to Q$ denotes the quotient map. The full separatrix set is
\[
 \Sigma=\bigcup_{A_L\cap S\ne\varnothing}A_L.
\]
\end{definition}

\Needspace{22\baselineskip}
\begin{theorem}[Global Hausdorff quotient]\label{thm:main}
Let $(M,\omega)$ be a connected compact K\"ahler manifold of complex dimension $n$, and let $\F$ be a possibly singular holomorphic foliation of complex leaf dimension $k$, where $0<k<n$. Let $S=\Sing(\F)$, let $\Sigma$ be the union of all regular leaf closures meeting $S$, and let $\sim$ be the equivalence relation generated by intersecting leaf closures in $M$. Suppose that
\begin{enumerate}[label=\textup{(\roman*)}]
\item there is $C<\infty$ such that $\Vol_{2k}(L)\leq C$ for every regular leaf;
\item $\dim_{\CC}S<k$;
\item $\HH^{2k}(\Sigma)<\infty$, for the metric induced by $\omega$.
\end{enumerate}
Then $\Sigma$ is a compact pure $k$-dimensional analytic subset, unless it is empty, and it contains $S$. Its connected components are precisely the equivalence classes meeting $S$. Every other equivalence class is a compact regular leaf. The relation $\sim$ is closed in $M\times M$, and $Q$ is compact Hausdorff.
\end{theorem}

Condition (i) bounds the volume of each entire leaf. Condition (iii) requires finite measure of the union of all leaf closures meeting $S$.

These hypotheses retain the two volume assumptions of the local Alexander--Verjovsky theorem \cite{AV}, with the isolated-singularity assumption replaced by $\dim_{\CC}S<k$. Theorem~\ref{thm:analytic} gives $Q$ a normal complex analytic structure.

When $S$ is empty, the two additional assumptions are vacuous. The conclusion agrees with the bounded-volume criterion for compact regular foliations in \cite{Epstein,EMS}. Section 9 of Mart\'inez--Verjovsky \cite{VM} treats regular holomorphic foliations on compact K\"ahler manifolds. A limit of leaf closures may be reducible. The inequality $\dim S<k$ ensures that each component is itself a leaf closure, and connectedness joins these components by a chain of intersections.

Reiffen \cite{Reiffen} uses the same chains of intersecting closures to define singular leaf spaces. His codimension-one integrability criterion is discussed in Remark~\ref{rem:reiffen}.

Malgrange's singular Frobenius theorems \cite{MalgrangeI,MalgrangeII} give local first integrals under codimension assumptions on the singular set of a Pfaff system. For a codimension-one foliation with $\operatorname{codim}S\ge3$, these local first integrals and the uniform volume bound give a global quotient onto a compact Riemann surface (Corollary~\ref{cor:malgrange-global}).

\section{Bishop's theorems and analytic leaf closures}

The following forms of Bishop's results are used; see \cite{Bishop,Chirka,Stolzenberg,VM} and the statements Bishop~1 and Bishop~3 in \cite[Section~3, pp.~4--5]{AV}. Local Euclidean volume and the volume of a smooth Hermitian metric are comparable on relatively compact coordinate charts.

\begin{theorem}[Finite-volume extension]\label{thm:extension}
Let $B$ be an analytic subset of a complex manifold $U$, and let $V$ be a closed pure $k$-dimensional analytic subset of $U\setminus B$. If $V$ has locally finite $2k$-volume up to $B$, then its closure in $U$ is a pure $k$-dimensional analytic subset.
\end{theorem}

\begin{theorem}[Compactness]\label{thm:bishop}
Let $A_j$ be nonempty compact pure $k$-dimensional analytic subsets of a compact Hermitian manifold, with uniformly bounded volumes. A subsequence converges in the Hausdorff topology to a nonempty compact pure $k$-dimensional analytic subset $A$. If each $A_j$ is connected, then $A$ is connected.
\end{theorem}

A Hausdorff limit of nonempty connected compact sets is connected. The analytic assertion follows from the local Bishop compactness theorem and a finite coordinate cover. When limits are expressed as analytic cycles, multiplicities may appear; the Hausdorff limit is the support of the limiting cycle.

Reiffen's Proposition~1.12 \cite[p.~275]{Reiffen} gives the corresponding leafwise criterion: local finiteness of the $2k$-dimensional Hausdorff measure of a regular leaf in the ambient manifold is equivalent to its closure being a proper separatrix in his terminology.

\begin{lemma}\label{lem:closures}
If a regular leaf $L$ has finite volume, then it is closed in $M^\circ$, and $A_L$ is a compact irreducible pure $k$-dimensional analytic subset of $M$. Moreover,
\[
 A_L\setminus S=L,\qquad \Vol_{2k}(A_L)=\Vol_{2k}(L).
\]
This statement does not require $\dim_{\CC}S<k$.
\end{lemma}

\begin{proof}
Suppose $p\in\overline L\cap M^\circ$ and $p\notin L$. Take a foliation chart with coordinates $(z,w)\in\Delta^k\times\Delta^{n-k}$ around $p$, and a smaller box with compact closure in it. A plaque is a set with $w$ constant. The leaf must meet infinitely many distinct plaques whose transverse parameters approach that of $p$. Indeed, if only finitely many plaques were involved, a convergent subsequence of points on one of them would place $p$ on that same plaque and hence on $L$.

The portions of these plaques over a fixed smaller $z$-polydisc are pairwise disjoint and have a common positive lower volume bound. This follows from uniform comparison of the coordinate metric with the Hermitian metric on the smaller box. Their union has infinite volume, a contradiction. Thus $L$ is closed in $M^\circ$. The same argument also shows that locally only finitely many plaques of $L$ meet a smaller box, so $L$ is an embedded analytic submanifold of $M^\circ$.

Apply Theorem~\ref{thm:extension} in coordinate neighborhoods of $S$. This proves analyticity and pure dimension of $A_L$. Compactness follows from compactness of $M$. A connected complex manifold is irreducible as an analytic space; since $L$ is dense in $A_L$, the latter is irreducible. The analytic subset $A_L\cap S$ is proper in $A_L$, and therefore has dimension at most $k-1$ and zero $2k$-dimensional measure. It follows that adjoining it does not change volume.
\end{proof}

\begin{lemma}\label{lem:limit}
Assume (i) and (ii) of Theorem~\ref{thm:main}. Every Hausdorff limit of leaf closures is a connected finite union of leaf closures. In particular, the entire limit is contained in one equivalence class for $\sim$.
\end{lemma}

\begin{proof}
Let $A_{L_j}\to A$. By Lemma~\ref{lem:closures} and Theorem~\ref{thm:bishop}, $A$ is compact, connected, analytic, and pure $k$-dimensional. It has finitely many irreducible components,
\[
 A=A_1\cup\cdots\cup A_r.
\]
Finiteness follows from local finiteness of the irreducible decomposition and compactness.

On $M^\circ$, the set $A$ is saturated by plaques. To see this, take a point $p\in A\cap M^\circ$, approximating points $p_j\in A_{L_j}$ in a foliation chart, and write their transverse coordinates as $w_j\to w$. The corresponding plaques are contained in $A_{L_j}$ and converge to the plaque through $p$. Thus that plaque is contained in $A$. Since $A$ is analytic of dimension $k$, it is locally a finite union of plaques. In particular, it is nonsingular away from $S$, and distinct irreducible components can meet only in $S$.

The inequality $\dim S<k$ ensures that no $A_i$ is contained in $S$. The complement of a proper analytic subset in an irreducible complex analytic space is connected. Hence $A_i\setminus S$ is connected. Locally it is a plaque, so it is contained in one regular leaf. Conversely, saturation and continuation along that leaf show that the whole leaf lies in $A_i\setminus S$: it cannot pass to another component at a regular point. Therefore
\[
 A_i=\overline{L_i}
\]
for a regular leaf $L_i$.

The intersection graph of the finite family $A_1,\ldots,A_r$ is connected. Otherwise $A$ would be the disjoint union of two nonempty compact sets. A path in this graph gives a finite chain of intersecting leaf closures between any two components. All points of $A$ are therefore equivalent.
\end{proof}

\begin{corollary}\label{cor:covered}
Under (i) and (ii), every point of $S$ belongs to a leaf closure meeting $S$. Consequently $S\subset\Sigma$.
\end{corollary}

\begin{proof}
Given $s\in S$, choose $p_j\in M^\circ$ tending to $s$ and let $L_j$ be their leaves. Pass to a Hausdorff-convergent subsequence of $A_{L_j}$. Its limit contains $s$. Lemma~\ref{lem:limit} expresses the limit as a finite union of leaf closures, one of which contains $s$.
\end{proof}

\section{Finite measure of the separatrix set}

\begin{proposition}\label{prop:sigma}
Under (i)--(iii), the set $\Sigma$ --if not empty-- is a compact analytic set of pure dimension $k$. It is a finite union of separatrix closures. Each connected component of $\Sigma$ is exactly one equivalence class for $\sim$.
\end{proposition}

\begin{proof}
On $M^\circ$, the set $\Sigma$ is saturated. In a foliation chart, every plaque it meets is wholly contained in it. In a smaller box, each such plaque contributes a fixed positive minimum amount of $2k$-volume. The finiteness of $\HH^{2k}(\Sigma)$ therefore implies that only finitely many plaques occur there. It follows that $\Sigma\setminus S$ is closed and analytic in $M^\circ$, of pure dimension $k$.

Its volume is locally finite up to $S$, so Bishop extension makes
\[
 Z=\overline{\Sigma\setminus S}
\]
an analytic subset of $M$. There are no additional points in $M^\circ$, because $\Sigma\setminus S$ is already closed there. Any additional points lie in $S$, which is contained in $\Sigma$ by Corollary~\ref{cor:covered}. Conversely, every point of $\Sigma$ lies in the closure of one of its regular leaves. Thus $Z=\Sigma$.

Write its finite irreducible decomposition as $\Sigma=A_1\cup\cdots\cup A_N$. As in Lemma~\ref{lem:limit}, each $A_i\setminus S$ is one regular leaf and $A_i$ is its closure. Each $A_i$ meets $S$: a regular point on it belongs, by the definition of $\Sigma$, to a leaf whose closure meets $S$.

The connected components of $\Sigma$ are the unions indexed by the connected components of the intersection graph of the $A_i$. Hence points in any one connected component are joined by a finite chain of leaf closures. No chain can leave $\Sigma$. Indeed, a leaf closure outside this family is disjoint from $S$ and is a compact regular leaf by Lemma~\ref{lem:closures}. It cannot intersect $\Sigma$ at a regular point, because distinct regular leaves are disjoint, and it cannot intersect it at a singular point because it avoids $S$. This also proves that every equivalence class outside $\Sigma$ is a single compact regular leaf.
\end{proof}

\section{Closedness of the equivalence relation}

\begin{proof}[Proof of Theorem~\ref{thm:main}]
Proposition~\ref{prop:sigma} gives the assertions concerning $\Sigma$ and the equivalence classes. It remains to prove closedness of
\[
 R=\{(x,y)\in M\times M:x\sim y\}.
\]
Suppose $(x_j,y_j)\in R$ and $(x_j,y_j)\to(x,y)$. There are two cases after passage to a subsequence.

First, suppose the equivalence class of $(x_j,y_j)$ meets $S$ for infinitely many $j$. There are only finitely many such classes, namely the connected components of $\Sigma$. On a further subsequence one fixed component contains both $x_j$ and $y_j$. It is closed, so it contains both limits, and $x\sim y$.

Otherwise, $x_j$ and $y_j$ lie on the same compact regular leaf $L_j$ disjoint from $\Sigma$. Bishop compactness gives a subsequence $L_j\to A$. Both $x$ and $y$ belong to $A$. By Lemma~\ref{lem:limit}, all points of $A$ belong to one equivalence class. Again $x\sim y$. Since $M\times M$ is metrizable, this proves that $R$ is closed.

To prove that $Q$ is Hausdorff, first observe that $q$ is closed. If $F\subset M$ is closed, its saturation is the projection of
\[
 R\cap(M\times F),
\]
which is compact and hence closed. Thus the quotient map $q:M\to Q$ is closed. Distinct equivalence classes $K_1,K_2$ are disjoint compact subsets of $M$. Choose disjoint open neighborhoods $O_1,O_2$ of them. The sets
\[
 V_i=Q\setminus q(M\setminus O_i),\qquad i=1,2,
\]
are open neighborhoods of the corresponding quotient points, and $q^{-1}(V_i)\subset O_i$. They are disjoint. Finally $Q$ is compact as a continuous image of $M$.
\end{proof}

\begin{proposition}[A weaker geometric assumption]\label{prop:weaker}
Assume (i) and (ii), and replace (iii) by the requirement that there are only finitely many equivalence classes meeting $S$, each of which is closed in $M$. Then $M/{\sim}$ is compact Hausdorff.
\end{proposition}

\begin{proof}
Corollary~\ref{cor:covered} still applies. A class avoiding $S$ is a single compact regular leaf, since any intersection of distinct leaf closures lies in $S$. The two-case sequential proof above therefore applies without change.
\end{proof}

Finite measure of $\Sigma$ implies the finiteness and closedness required in Proposition~\ref{prop:weaker}. It is not necessary for Hausdorffness: the pencil of lines through a point in $\PP^2$ has bounded leaf areas and quotient a single point, although $\Sigma=\PP^2$ has infinite $\HH^2$-measure. Whether (i) and (ii) alone imply Hausdorffness is not settled here.

\begin{remark}
The proof of Theorem~\ref{thm:main}, with the volume bound assumed, uses only a smooth Hermitian metric on a compact complex manifold. Closedness of the K\"ahler form is used in the homological volume computations in the examples below.
\end{remark}

\section{The global complex analytic quotient}\label{sec:analytic}

All complex analytic spaces are understood to be Hausdorff. For regular holomorphic leaf quotients see Holmann \cite{Holmann}; for analytic quotients and invariant holomorphic functions see Cartan \cite{Cartan}.

\begin{theorem}[Global analytic quotient]\label{thm:analytic}
Under hypotheses \textup{(i)--(iii)} of Theorem~\ref{thm:main}, the space $Q=M/{\sim}$ has a canonical structure of a normal compact complex analytic space of dimension $n-k$. The projection
\[
 q:M\longrightarrow Q
\]
is a proper surjective holomorphic map with connected, pure $k$-dimensional fibers, and
\[
 \mathcal O_Q=q_*\mathcal O_M.
\]
Its fibers are exactly the prescribed finite-chain classes of leaf closures. In addition, $R=M\times_Q M$, taken with its reduced structure, is a pure $(n+k)$-dimensional analytic subset of $M\times M$.
\end{theorem}

A compact effective $k$-cycle is a finite sum $\sum_i m_i Z_i$, where the $Z_i$ are irreducible compact $k$-dimensional analytic sets and $m_i$ are positive integers. Its support is $\bigcup_i Z_i$. Barlet's space $\mathcal C_k(X)$ parametrizes such compact cycles and has a complex analytic structure \cite{Barlet}.

\begin{theorem}[Analytic family of fiber cycles]\label{thm:cycles}
Let $T$ be a connected normal complex space and $X$ a reduced complex space, and let $G\subset T\times X$ be a reduced analytic subset proper over $T$. Suppose that every irreducible component of $G$ maps surjectively to $T$ with pure $k$-dimensional fibers. There is a holomorphic map
\[
 \Phi:T\longrightarrow\mathcal C_k(X)
\]
whose value at $t$ has support $G_t=\{x:(t,x)\in G\}$. The multiplicities are uniquely specified by requiring the generic fibers to have multiplicity one.
\end{theorem}

This is the proper version of Barlet's fiber-cycle construction; see \cite[Proposition~1, p.~1161]{Mathieu}, which states the normal-base, constant-corank criterion. Here constant corank means that the local dimension of every fiber is $k$. Since the graph is proper over the base, the analytic family defines a holomorphic map to the space of compact cycles. The cycles are reduced on a dense open subset of the base; at other parameters their multiplicities may be greater than one. Two cycles with the same support need not be equal. Uniqueness follows from the generic normalization and continuity into the Hausdorff cycle space. For Remmert's proper mapping theorem and Stein factorization, see \cite{Fischer}.

\begin{proof}[Proof of Theorem~\ref{thm:analytic}]
Put $U=M\setminus\Sigma$. By Theorem~\ref{thm:main}, $\Sigma$ is either empty or a proper compact pure $k$-dimensional analytic set, $S\subset\Sigma$, and $U$ is saturated by compact regular leaves. Every equivalence class is a connected compact pure $k$-dimensional analytic set. $R$ is closed in $M\times M$.

\smallskip\noindent\emph{1. The regular relation.}
The uniform volume bound implies finite holonomy for every leaf in $U$ by the compact-leaf criterion \cite[Theorems~4.2--4.3]{Epstein}. Reeb stability and the holomorphic transverse structure give local leaf quotients $T/H$, where $T$ is a transverse complex disc of dimension $n-k$ and $H$ is the finite holonomy group. Consequently
\[
 R_0=R\cap(U\times U)
\]
is analytic of pure dimension $n+k$. Near a pair of points on one leaf its branches are given in foliation coordinates by equations $w'=h(w)$ for finitely many holomorphic transverse transformations $h$; the two longitudinal coordinates are free. These descriptions also show that both projections $R_0\to U$ are open.

\smallskip\noindent\emph{2. Extending the relation across the separatrix set.}
Set $G=\overline{R_0}^{\,M\times M}$. Closedness of $R$ gives $G\subset R$. No equivalence class has one point in $\Sigma$ and one point in $U$. Since $R$ is closed, each mixed pair has a neighborhood disjoint from $R$. Thus
\[
 G\setminus R_0\subset\Sigma\times\Sigma.
\]
It follows that $R_0$ is a closed analytic subset of $(M\times M)\setminus(\Sigma\times\Sigma)$: near mixed pairs it is empty. Since
\[
 \dim(\Sigma\times\Sigma)=2k<n+k=\dim R_0,
\]
the Remmert--Stein extension theorem makes $G$ a compact pure $(n+k)$-dimensional analytic subset of $M\times M$. If $\Sigma$ is empty, take $G=R_0$. The diagonal of $M$ lies in $G$, because the diagonal of the dense open set $U$ lies in $R_0$.

Let $G_i$ be an irreducible component of $G$, endowed with its reduced structure, and let $p:G_i\to M$ be either projection, these projections are both proper. For $a\in M$, the fiber $p^{-1}(a)$ is contained in the class $R_a$, so every irreducible component of this fiber has dimension at most $k$. For the second projection this uses the symmetry of $R$. Compare \cite{Mathieu}.

For the reverse inequality, take a point $z\in p^{-1}(a)$ and local coordinates $t_1,\ldots,t_n$ on $M$ centered at $a$. The fiber is cut out in the pure $(n+k)$-dimensional space $G_i$ by the $n$ holomorphic equations
\[
 t_1\circ p=\cdots=t_n\circ p=0.
\]
Every local irreducible component of their common zero set therefore has dimension at least $(n+k)-n=k$. Hence every nonempty fiber of $p$ is pure $k$-dimensional. Remmert's theorem makes $p(G_i)$ an analytic subset of $M$, and the dimension theorem gives
\[
 \dim p(G_i)=\dim G_i-k=n.
\]
A connected complex manifold has no proper analytic subset of its own dimension, so $p(G_i)=M$.

Thus each component of $G$ is proper and surjective over $M$, with constant corank $k$. The base is the smooth manifold $M$, which is normal. Theorem~\ref{thm:cycles} therefore applies with base $M$.

\smallskip\noindent\emph{3. The cycle map and its fibers.}
Apply theorem \ref{thm:cycles} to the first projection of $G$. It gives a holomorphic map
\[
 \Phi:M\longrightarrow\mathcal C_k(M),\qquad
 |\Phi(x)|=G_x\subset R_x,
\]
where $R_x$ is the equivalence class of $x$ and $|\cdot|$ denotes cycle support. Choose a dense open set $V\subset M$ on which the cycles $\Phi(x)$ are reduced. Let $x\in U$. By Theorem~\ref{thm:main}, the class $R_x$ is the compact
leaf $L_x$, and the inclusions $R_0\subset G\subset R$ give
$G_x=L_x$. If moreover $x\in V$, the cycle $\Phi(x)$ is reduced, so
$\Phi(x)=[L_x]$. In particular, for $(x,y)\in R_0$ with $x,y\in V$,
\[
 \Phi(x)=[L_x]=[L_y]=\Phi(y).
\]
Put $V_U=V\cap U$. Each projection $p_j:R_0\to U$ is open and surjective. If $W\subset R_0$ is nonempty and open, then $p_j(W)$ is nonempty and open in $U$, and hence meets $V_U$. Thus $p_j^{-1}(V_U)$ is dense and open in $R_0$. Their intersection
\[
 D=p_1^{-1}(V_U)\cap p_2^{-1}(V_U)
\]
is dense in $R_0$. On $D$ the two cycles are equal, including their multiplicities, since both are the reduced cycle of the same leaf. The cycle space is Hausdorff, so the equality locus of the two continuous maps $\Phi\circ p_1$ and $\Phi\circ p_2$ is closed. Equality therefore holds throughout $R_0$ and then throughout its closure $G$. We obtain
\begin{equation}\label{eq:equal-on-G}
 (x,y)\in G\quad\Longrightarrow\quad\Phi(x)=\Phi(y).
\end{equation}

Every leaf closure $A_L$ satisfies $A_L\times A_L\subset G$. To verify this, take $x,y\in L$ and a path in $L$ joining them. A finite chain of foliation chartes along the path supplies a holomorphic transport map between small transversals through $x$ and $y$. Choose points $x_j\in U$ on the first transversal tending to $x$, and transport them to points $y_j\to y$. Indeed, in a foliation chart $\Sigma\cap M^\circ$ is a finite union of plaques. Its intersection with a sufficiently small transversal is therefore finite, and its complement is dense because $n-k>0$. Saturation of $U$ ensures that the transported points also belong to $U$. The points $x_j,y_j$ lie on the same leaf in $U$, so $(x_j,y_j)\in R_0$. Hence $(x,y)\in G$. Closing in each factor proves the assertion for $A_L$.

By \eqref{eq:equal-on-G}, $\Phi$ is constant on each $A_L$. It is therefore constant on every finite chain of intersecting leaf closures, hence on every $R$-class. Conversely, if $\Phi(x)=\Phi(y)$, the diagonal inclusion gives
\[
 x\in G_x=|\Phi(x)|=|\Phi(y)|=G_y\subset R_y,
\]
so $x\sim y$. We have proved
\begin{equation}\label{eq:fibers}
 \Phi(x)=\Phi(y)\quad\Longleftrightarrow\quad x\sim y.
\end{equation}
This also shows $G=R$. If $x\sim y$, then $\Phi(x)=\Phi(y)$ and hence $G_x=G_y$. The diagonal inclusion gives $y\in G_y$, so $y\in G_x$ and $(x,y)\in G$. Together with $G\subset R$, this proves equality.

\smallskip\noindent\emph{4. The normal analytic structure.}
The map $\Phi$ is proper because $M$ is compact and the cycle space is Hausdorff. Its image $Z=\Phi(M)$, with its reduced structure, is a compact complex analytic space by Remmert's theorem. Take its Stein factorization
\[
 M\xrightarrow{q_{\mathrm{an}}}Q_{\mathrm{an}}
   \xrightarrow{\nu}Z.
\]
Since $M$ is normal, $Q_{\mathrm{an}}$ is normal, $q_{\mathrm{an}}$ has connected fibers, and $(q_{\mathrm{an}})_*\mathcal O_M=\mathcal O_{Q_{\mathrm{an}}}$, see \cite{Serre}. Every fiber of $\Phi$ is connected by \eqref{eq:fibers} and Theorem~\ref{thm:main}. Consequently the finite map $\nu$ is bijective and the fibers of $q_{\mathrm{an}}$ are exactly the $R$-classes. The induced continuous bijection $M/{\sim}\to Q_{\mathrm{an}}$ is a homeomorphism, by compactness and Hausdorffness. Transfer the normal analytic structure through it. The identity $\mathcal O_Q=q_*\mathcal O_M$ characterizes the resulting structure intrinsically. Finally the fibers have pure dimension $k$, so $\dim Q=n-k$.
\end{proof}

\begin{remark}[Cartan and invariant holomorphic functions]
Cartan's work on analytic quotients \cite{Cartan} relates the quotient structure to holomorphic functions constant on equivalence classes. In the present construction this description follows from Stein factorization. For every open set $W\subset Q$,
\[
 \mathcal O_Q(W)\xrightarrow{\ q^*\ }
 \mathcal O_M(q^{-1}(W))
\]
is an isomorphism. Indeed, the reduced analytic sets underlying the fibers are compact and connected, so every holomorphic function on $q^{-1}(W)$ is constant on each fiber, and the direct-image identity in Step~4 gives its holomorphic descent. Local coordinates in an analytic embedding of $W$ pull back to invariant holomorphic functions separating the fibers over $W$. For a finite holonomy group $H$, the transverse chart $T/H$ has the invariant structure sheaf $(\pi_*\mathcal O_T)^H$, where $\pi:T\to T/H$ is the orbit map.
\end{remark}

\begin{remark}[Relation with Reiffen's leaf spaces]\label{rem:reiffen}
In \cite[Section~3, p.~280]{Reiffen}, a family of regular leaves is called connected when any two members can be joined by a finite sequence of leaves with intersecting closures. This is the same finite-chain condition used in the definition of $\sim$. Reiffen's conditions~3.1 concern the covering of the singular set by leaf closures, closedness of regular leaves outside the singular set, and analytic extension of those leaves. Lemma~\ref{lem:closures} and Corollary~\ref{cor:covered} establish these properties under hypotheses~\textup{(i)} and~\textup{(ii)} of Theorem~\ref{thm:main}.

Reiffen also considers the family
\[
 \mathcal L_{\mathrm{int}}
 =\{L:\text{there is a regular leaf }L'\ne L
           \text{ with }A_L\cap A_{L'}\ne\varnothing\},
\]
denoted by $L^*$ in \cite[Section~3]{Reiffen}. Under hypothesis~\textup{(iii)}, this family is finite: two distinct leaf closures can intersect only in $S$, and Proposition~\ref{prop:sigma} gives only finitely many leaf closures meeting $S$.

For a codimension-one foliation, Reiffen's Proposition~4.7 \cite[pp.~284]{Reiffen} states that closedness of the regular leaves outside the singular set, finiteness of $\mathcal L_{\mathrm{int}}$, and Hausdorffness of the space obtained by identifying maximal connected families yield a holomorphic first-integral map to a Riemann surface. Theorem~\ref{thm:main} establishes Hausdorffness of the quotient on $M$ from the volume and dimension hypotheses, and Theorem~\ref{thm:analytic} constructs its normal analytic structure in arbitrary codimension. Reiffen's Definition~2.13 also equips the leaf space with the sheaf of functions whose pullbacks are holomorphic, as in the identity $\mathcal O_Q=q_*\mathcal O_M$ above.
\end{remark}

\section{Separatrices, holonomy, and singularities of the quotient}\label{sec:singularities}

Write $r=n-k$ and $B=q(\Sigma)$. The set $B$ is finite, because $\Sigma$ has finitely many connected components. The inverse image $q^{-1}(B)$ equals $\Sigma$. The points of $B$ correspond to classes meeting $S$.

\begin{proposition}[Local holonomy quotient]\label{prop:holonomy}
For a compact regular leaf $L\subset M\setminus\Sigma$, let $H_L$ denote its effective finite holonomy group acting on a transverse germ $(\CC^r,0)$. Then
\[
 (Q,q(L))\simeq(\CC^r/H_L,0).
\]
The action of $H_L$ is holomorphically conjugate to its derivative representation
$\rho:H_L\to\operatorname{GL}(r,\CC)$, where $\rho(h)=D_0h$.
\end{proposition}

\begin{proof}
Step~1 of the proof of Theorem~\ref{thm:analytic} gives the transverse quotient germ. Its analytic structure agrees with that of $Q$: both have the sheaf of holomorphic functions whose pullbacks are holomorphic and constant on the leaves.

Choose transverse coordinates centered at the fixed point and define the holomorphic germ
\[
 F(z)=\frac{1}{|H_L|}\sum_{h\in H_L}\rho(h)^{-1}h(z).
\]
Each summand has derivative $I$ at the origin, so $D_0F=I$. For $g\in H_L$, reindexing the sum by $u=hg$ gives
\[
 F(g(z))
 =\frac{1}{|H_L|}\sum_{u\in H_L}\rho(g)\rho(u)^{-1}u(z)
 =\rho(g)F(z).
\]
The holomorphic inverse function theorem makes $F$ a local change of coordinates conjugating the action to $\rho$. Since the original action is effective, this identity also shows that $\rho$ is faithful.
\end{proof}

A \emph{complex pseudoreflection} is a nonidentity linear transformation fixing a hyperplane pointwise. The trivial group is regarded as generated by the empty set.

\begin{corollary}[Smoothness criterion]\label{cor:holonomy-smoothness}
Identify $H_L$ with its derivative representation. The quotient $Q$ is smooth at $q(L)$ if and only if $H_L$ is generated by complex pseudoreflections. Consequently,
\[
 \Sing Q\setminus B
 =\{q(L):L\subset M\setminus\Sigma,\
 H_L\text{ is not generated by pseudoreflections}\}.
\]
\end{corollary}

\begin{proof}
By Proposition~\ref{prop:holonomy}, the germ at $q(L)$ is the quotient of $(\CC^r,0)$ by the finite linear group $H_L$. The local smoothness form of the Chevalley--Shephard--Todd theorem states that this quotient germ is smooth exactly when the group is generated by complex pseudoreflections \cite{Chevalley,ShephardTodd}. Its equivalent invariant-polynomial formulation states that $\CC[z_1,\ldots,z_r]^{H_L}$ is a polynomial algebra exactly for such groups; see \cite{Chevalley,ShephardTodd}. Applying the local criterion at each leaf gives the displayed equality.
\end{proof}

The finite holonomy charts define a complex orbifold structure on $Q\setminus B$. Its isotropy groups are the effective holonomy groups of the leaves. Nontrivial isotropy can occur at a smooth point of the underlying analytic space. An extension across $B$ requires compatible finite-group uniformizing charts at its points.

Thus, if every compact regular leaf outside $\Sigma$ has trivial holonomy, or more generally holonomy generated by pseudoreflections, then
\[
 \Sing Q\subset B.
\]
Even with this additional assumption equality need not hold. In complex codimension one $r=1$, normality implies that $Q$ is a compact Riemann surface, so $\Sing Q$ is empty regardless of the presence of separatrices.

\begin{example}[A singular quotient without separatrices]\label{ex:holonomy}
Let $E$ be an elliptic curve, let $\tau\in E$ be a nonzero point of order two, and let $A$ be an abelian surface. The involution
\[
 \iota:E\times A\longrightarrow E\times A,
 \qquad(e,a)\longmapsto(e+\tau,-a)
\]
acts freely, because translation by $\tau$ has no fixed point on $E$. Therefore $M=(E\times A)/\langle\iota\rangle$ is a smooth compact K\"ahler threefold and we can choose the K\"ahler form to be invariant by $\iota$. The product foliation in the $E$ direction descends to a nonsingular holomorphic foliation of rank one. In particular, $S=\Sigma=\varnothing$.

Write $A[2]=\{a\in A:2a=0\}$ for the two-torsion subgroup. For $a\notin A[2]$, the leaf associated with $\{a,-a\}$ is an elliptic curve isomorphic to $E$, with area $\Vol(E)$. For $a\in A[2]$, its area is $\Vol(E)/2$, since the leaf is $E/\langle\tau\rangle$. Hence the leaf areas are uniformly bounded and all three hypotheses hold.

The leaf quotient is $Q=A/\{\pm1\}$. At each of the sixteen points of $A[2]$, the transverse holonomy is $\{\pm I\}\subset\operatorname{GL}(2,\CC)$. Its nonidentity element is not a pseudoreflection. The local invariant coordinates
\[
 u=z_1^2,\qquad v=z_1z_2,\qquad w=z_2^2
\]
satisfy $uw=v^2$, an ordinary double point. Thus $Q$ has sixteen singular points, although the foliation has no singularities and no separatrices.
\end{example}

The blown-up fibration in Section~7 has $\Sigma=C\cup E$ and smooth quotient $\PP^1$. Thus neither inclusion between $\Sing Q$ and $q(\Sigma)$ holds in general.

\section{Non-Hausdorff leaf spaces}

\subsection{A local normal crossing}
In the unit polydisc of $\CC^3$, the foliation
\[
 d(xy)=y\,dx+x\,dy=0
\]
has singular set $\{x=y=0\}$. Its regular leaves include $L_c=\{xy=c\}$ for $c\ne0$, together with the two distinct leaves $\{x=0,y\ne0\}$ and $\{y=0,x\ne0\}$. For $0<|c|<1$, the curve $xy=c$ in the bidisc is parametrized by $x\mapsto(x,c/x)$, with $|c|<|x|<1$. Its Euclidean area is
\[
 \int_{|c|<|x|<1}\left(1+\frac{|c|^2}{|x|^4}\right)dA(x)
 =2\pi(1-|c|^2)\leq2\pi.
\]
The product with the $z$-disc has four-dimensional volume at most $2\pi^2$. For fixed $a\ne0$ sufficiently small,
\[
 (a,c/a,0)\longrightarrow(a,0,0),\qquad
 (c/a,a,0)\longrightarrow(0,a,0)
\]
as $c\to0$. The two points on the left belong to the same leaf; the limits belong to different leaves. The ordinary leaf space is non-Hausdorff. The two limiting leaves belong to the same fiber $\{xy=0\}$. To get a compact manifold, notice that this local foliation defined in the unit polydisk can be extended to $\PP^3$.

\subsection{A compact projective example}
Blow up a point $p$ on the fiber over $0$ of the first projection of $\PP^1\times\PP^1$. The resulting surface $M$ is smooth and projective. The induced morphism
\[
 f:M\longrightarrow\PP^1
\]
has smooth rational fibers $F_t$ for $t\ne0$, and one reducible fiber
\[
 F_0=C\cup E,
\]
where $C$ is the strict transform of the original fiber and $E$ is the exceptional curve. They meet transversely at one point $q$, and locally $f(u,v)=uv$. The foliation by fibers has the distinct regular leaves $C\setminus\{q\}$ and $E\setminus\{q\}$.

For any K\"ahler form $\omega$,
\[
 \int_{F_t}\omega=\int_C\omega+\int_E\omega\qquad(t\ne0),
\]
so all regular leaf areas are uniformly bounded. The preceding local argument proves that the ordinary regular-leaf space is non-Hausdorff. Under $\sim$, the fiber $C\cup E$ is one class and the quotient is $\PP^1$. This example satisfies Theorem~\ref{thm:main}: $k=1$, $S=\{q\}$, and $\Sigma=C\cup E$. Taking a product with a compact K\"ahler manifold $Z$ and using the fibers of $f\circ\operatorname{pr}_M$ gives nonisolated singularities with $\dim S=\dim Z<1+\dim Z=k$.

\section{A counterexample for chains of intersecting closures}\label{sec:counterexample}

The following example has bounded leaf volumes and a nonclosed relation $\sim$. A limit of leaf closures contains a component supported in $S$.

Take
\[
 M=\PP^1_t\times\PP^1_x\times\PP^1_y
\]
with product K\"ahler form $\omega=\omega_t+\omega_x+\omega_y$, each factor having area $1$. On the finite $t$-chart define
\[
 X_t=\frac{\partial}{\partial x}+t\frac{\partial}{\partial y}.
\]
The vector field $\partial/\partial x$ is global on $\PP^1_x$: in the coordinate $u=1/x$ it is $-u^2\partial/\partial u$. On the base chart $\tau=1/t$ use
\[
 Y_\tau=\tau\frac{\partial}{\partial x}+\frac{\partial}{\partial y}.
\]
On the overlap $Y_\tau=\tau X_t$, so the line fields agree. They define an integrable rank-one holomorphic subsheaf of $TM$. The singular set is the union of the three projective lines:
\begin{align*}
 S={}&\bigl(\PP^1_t\times\{\infty\}\times\{\infty\}\bigr)\\
 &\cup\bigl(\{0\}\times\{\infty\}\times\PP^1_y\bigr)\\
 &\cup\bigl(\{\infty\}\times\PP^1_x\times\{\infty\}\bigr).
\end{align*}
Which clearly is of codimension two. The local generators have no common divisorial factor in their coefficients, so their rank-one image is saturated. For example, near $t=0$, $x=\infty$, and finite $y$, a generator is $-u^2\partial_u+t\partial_y$. Thus the singularities persist after saturation.

For $t\ne0,\infty$, the affine leaves are
\[
 L_{t,c}=\{t\}\times\{y=tx+c\},\qquad c\in\CC.
\]
Their closures $C_{t,c}$ are smooth curves of bidegree $(1,1)$ and all pass through $q_t=(t,\infty,\infty)$. The two lines at infinity, with $q_t$ deleted, are also regular leaves. All their closures intersect at $q_t$. Therefore the entire fiber over each $t\ne0,\infty$ is one equivalence class.

At $t=0$, the regular leaves and their closures are
\[
 L_{0,c}=\{0\}\times(\PP^1_x\setminus\{\infty\})\times\{c\},
 \qquad
 C_{0,c}=\{0\}\times\PP^1_x\times\{c\},
 \qquad c\in\PP^1_y.
\]
The $C_{0,c}$ are pairwise disjoint. A chain of leaf closures cannot change its $t$-coordinate. Hence each $C_{0,c}$ is a distinct equivalence class. At $t=\infty$, the analogous description has the roles of $x$ and $y$ reversed. Every point of $M$, including every singular point, belongs to a regular leaf closure.

Every curve $C_{t,c}$ of bidegree $(1,1)$ has area $2$. The remaining leaf closures have area $1$. Removing the finitely many missing points does not change area. Thus
\[
 \Vol_2(L)\leq2\quad\text{for every regular leaf }L.
\]

Choose $t_j\ne0,\infty$ with $t_j\to0$. Then
\[
 p_j=(t_j,0,0)\sim r_j=(t_j,0,1),
\]
because the two leaf closures meet at $q_{t_j}$. On the other hand,
\[
 p_j\to p=(0,0,0),\qquad r_j\to r=(0,0,1),
\]
and $p\not\sim r$. The relation is not closed, so its quotient is not Hausdorff. Both limiting points are regular points of the foliation.

Use homogeneous coordinates $x=X_1/X_0$ and $y=Y_1/Y_0$. The curve $C_{t,c}$ has equation
\[
 X_0Y_1-tX_1Y_0-cX_0Y_0=0.
\]
At $t=0$ this becomes $X_0(Y_1-cY_0)=0$. Accordingly, as analytic cycles,
\[
 C_{t,c}\longrightarrow C_{0,c}+D,
 \qquad D=\{0\}\times\{\infty\}\times\PP^1_y.
\]
The curve $D$ lies entirely in $S$ and is not itself the closure of a regular leaf. It meets every $C_{0,c}$, but these intersections do not generate identifications under the specified relation because $D$ is not one of its generating leaf closures. Here $\dim_{\CC}S=k=1$, whereas Lemma~\ref{lem:limit} requires $\dim_{\CC}S<k$.

Thus a uniform leaf-volume bound alone does not imply closedness of the finite-chain relation. Here $\Sigma=M$ and $\HH^2(\Sigma)=\infty$.

\section{The local theorem and codimension one}

The local Alexander--Verjovsky theorem \cite{AV} concerns a $k$-dimensional foliation near an isolated singularity, with uniformly bounded leaf volumes. For the analytic quotient conclusion it also assumes finite $\HH^{2k}$-measure of the full separatrix set. The quotient identifies the whole separatrix set to one point. Its initial proposition permits a nonisolated singular set of dimension less than $k$.

The proof of Lemma~\ref{lem:closures} also applies locally, across an arbitrary analytic singular set. The local closure is analytic, although it need not be compact.

\begin{proposition}\label{prop:MM}
Let $\F$ be a germ of a codimension-one holomorphic foliation at $0\in\CC^n$, represented on a neighborhood $U$, with singular set of codimension at least two. Put $k=n-1$. Suppose each regular leaf in $U$ has finite $2k$-volume, and the union
\[
 \Sigma_0=\bigcup_{0\in\overline L^{\,U}}\overline L^{\,U}
\]
has finite $\HH^{2k}$-measure near $0$. Then $\F$ admits a nonconstant holomorphic first integral near $0$.
\end{proposition}

\begin{proof}
Finite volume implies closedness of leaves off the singular set and analyticity of their closures by the local version of Lemma~\ref{lem:closures}. Choose a ball $\overline{B_r}\subset U$ centered at $0$ on which $\Sigma_0$ has finite measure. Distinct leaf closures are disjoint off the singular set. Since $\dim S\leq n-2<k$, their intersections have zero $2k$-measure.

Each leaf adhering to $0$ contributes a positive measure subset of $\Sigma_0\cap B_r$. A finite measure set contains at most countably many pairwise disjoint positive measure measurable subsets. For example, those having measure at least $1/m$ are finite for each $m$, and their union over $m$ includes every positive measure member. Thus at most countably many leaves adhere to $0$.

Theorem B of Mattei--Moussu \cite[p.~473]{MM} applies in $\CC^n$: an integrable holomorphic one-form admits a nonconstant holomorphic first integral if its leaves are closed outside the singular set and at most countably many leaves adhere to the origin. The singular set may be nonisolated.
\end{proof}

The local volume monotonicity inequality gives a finite bound on the number of these leaves. Each pure $k$-dimensional analytic leaf closure through $0$ satisfies
\[
 \Vol_{2k}(\overline L^{\,U}\cap B_r)
 \geq \frac{\pi^k}{k!}r^{2k}
\]
for Euclidean volume. Additivity off the measure-zero singular set then bounds the number of such leaves by the total volume of $\Sigma_0\cap B_r$ divided by this lower bound.

For a nonconstant holomorphic first integral $f$ on a connected neighborhood, the quotient by the full fibers of $f$ is Hausdorff: $f$ is open and induces a homeomorphism from that quotient onto its image in $\CC$. In this quotient, regular components of the same singular fiber are identified.

\section{Relation with Malgrange's singular Frobenius theorems}\label{sec:malgrange}

Malgrange \cite{MalgrangeI,MalgrangeII} proves the existence of holomorphic first integrals near singular points under codimension assumptions on an integrable Pfaff system. Ramis \cite{Ramis} gives an account of these results and of the work of Mattei and Moussu. In codimension one, Malgrange's theorem implies the separatrix-measure condition under the hypotheses of Corollary~\ref{cor:malgrange-global}.

Throughout this section $p=n-k$ is the codimension of the foliation. The letter $c$ denotes the codimension of a singular set, and $q$ continues to denote the global quotient map.

\subsection{First integrals and invertible integrating factors}

For a holomorphic foliation germ $\F$ at $a$, a holomorphic first integral is a nonconstant holomorphic germ $h$ whose differential annihilates the tangent directions of $\F$ at regular points. In codimension one, if $\F$ is defined by a holomorphic one-form $\alpha$, this condition is
\[
 dh\wedge\alpha=0.
\]
It is weaker than the existence of germs $u,f$ such that
\begin{equation}\label{eq:malgrange-unit}
 \alpha=u\,df,\qquad u(a)\ne0.
\end{equation}
A holomorphic germ nonzero at $a$ is called a unit; it has a holomorphic multiplicative inverse near $a$. In \eqref{eq:malgrange-unit}, $u^{-1}\alpha$ is exact. Malgrange proves the existence of such a unit factor. Ramis adopts the  stronger usual convention for a first integral or ``int\'egrale premi\`ere'' in Definition~0.2 of his exposition.

For a codimension-one foliation on a smooth manifold, one may choose a local defining one-form with no common divisorial factor in its coefficients. This is a saturated local equation. A first integral $h$ can have a differential with a common factor vanishing on a divisor, even when the saturated equation has a singular set of codimension two. Removing that common factor preserves the foliation but need not preserve a representation of the form \eqref{eq:malgrange-unit}.

\begin{theorem}[Malgrange, codimension one]\label{thm:malgrange-one}
Let $\alpha$ be a nonzero germ of holomorphic one-form at $0\in\CC^n$, with
\[
 \alpha\wedge d\alpha=0,
 \qquad S(\alpha)=\{x:\alpha(x)=0\}.
\]
If $\operatorname{codim}S(\alpha)\ge3$, then $\alpha=u\,df$ for holomorphic germs $u,f$ with $u(0)\ne0$. If $\operatorname{codim}S(\alpha)\ge2$, the same conclusion holds provided there is a formal representation $\alpha=\widehat u\,d\widehat f$ with $\widehat u(0)\ne0$.
\end{theorem}

These are Theorems~0.1 and~0.2 of \cite[p.~163]{MalgrangeI}. Formal germs belong to $\CC[[z_1,\ldots,z_n]]$, whereas holomorphic germs belong to the ring of convergent series $\CC\{z_1,\ldots,z_n\}$. The second assertion supplies holomorphic germs representing the form; it does not assert convergence of every previously chosen formal first integral.

\subsection{The algebraic and analytic steps in Malgrange's proof}

Let $\mathcal O=\CC\{z_1,\ldots,z_n\}$ and let $\Omega^j$ be the $\mathcal O$-module of holomorphic $j$-form germs. The Koszul complex associated with $\alpha$ is
\[
 K(\alpha):\quad
 0\longrightarrow\mathcal O\xrightarrow{\alpha\wedge}\Omega^1
 \xrightarrow{\alpha\wedge}\Omega^2\longrightarrow\cdots
 \xrightarrow{\alpha\wedge}\Omega^n\longrightarrow0.
\]
The differential in this complex is $\delta(\eta)=\alpha\wedge\eta$, whereas $d$ denotes exterior differentiation. Since $\alpha$ is a one-form, skew-symmetry gives $\alpha\wedge\alpha=0$, and therefore $\delta^2=0$. If the coefficients of $\alpha$ vanish on a set of codimension $c$, the cohomology of this complex satisfies
\[
 H^j(K(\alpha))=0\qquad(j<c);
\]
see \cite[Proposition~2.1]{MalgrangeI}. Thus $c\ge3$ gives $H^2(K(\alpha))=0$. The integrability condition $\alpha\wedge d\alpha=0$ says that $d\alpha$ is a degree-two cocycle for $\delta$. The vanishing of $H^2(K(\alpha))$ therefore yields a holomorphic one-form $\alpha_1$ with
\[
 d\alpha=\alpha\wedge\alpha_1.
\]
Differentiating gives $\alpha\wedge d\alpha_1=0$, hence $d\alpha_1=\alpha\wedge\alpha_2$ for another holomorphic one-form. The recursive construction gives forms $\alpha_0=\alpha,\alpha_1,\ldots$ satisfying
\[
 d\alpha_m=\sum_{j=0}^{m}\binom mj
                 \alpha_j\wedge\alpha_{m-j+1}.
\]
Such a sequence is called a Godbillon--Vey sequence.

Introduce one additional variable $t$ and consider
\[
 \Theta=dt+\sum_{j\ge0}\frac{t^j}{j!}\alpha_j.
\]
The identities give $\Theta\wedge d\Theta=0$, formally. The form $\Theta$ is nonsingular because its $dt$ coefficient is one. Once the series converges, ordinary Frobenius supplies a unit $U$ and a holomorphic $F$ with $\Theta=U\,dF$; restriction to $t=0$ proves \eqref{eq:malgrange-unit}.

To prove convergence, Malgrange constructs linear solution operators for the division equations with bounds on a common family of shrinking polydiscs. Cauchy estimates control differentiation, and a convergent scalar majorant controls the recursively chosen coefficients $\alpha_j/j!$. This turns the formal construction of Martinet and Moussu into an analytic one. In higher codimension, generalized Koszul complexes supply the division statements, and estimates uniform in the degree replace the one-form estimates; see \cite[Sections~1, 4--6]{MalgrangeII} and \cite[Sections~I, III--IV]{Ramis}.

\subsection{First integrals in codimension two}

For a codimension-one foliation, $k=n-1$, so the hypothesis $\dim S<k$ is precisely
\[
 \operatorname{codim}S\ge2.
\]
Malgrange gives a unit integrating factor wherever the local singular set has codimension at least three. At codimension-two singularities, Proposition~\ref{prop:MM} gives a holomorphic first integral from finite leaf volumes and finite measure of the separatrices through the point, using the closed-leaf criterion of Mattei--Moussu. The latter conclusion need not include a unit integrating factor.

\begin{example}[A first integral without a unit integrating factor]\label{ex:unit-factor}
On the bidisc $\Delta^2\subset\CC^2$, consider
\[
 \alpha=2y\,dx+x\,dy,
 \qquad h=x^2y.
\]
The coefficients of $\alpha$ have no common factor, $S(\alpha)=\{0\}$, and
\[
 dh=x\alpha.
\]
Thus $h$ is a holomorphic first integral. Nevertheless $\alpha$ has no representation $\alpha=u\,df$ with $u(0)\ne0$: such a representation would give $df(0)=0$, hence
\[
 d\alpha(0)=du(0)\wedge df(0)=0,
\]
whereas $d\alpha=-dx\wedge dy$. The same calculation excludes a formal unit integrating factor.

This example satisfies the local volume assumptions. For $0<|t|<1$, the leaf $h=t$ is parametrized by
\[
 x\longmapsto (x,t/x^2),\qquad \sqrt{|t|}<|x|<1.
\]
Its Euclidean area is
\begin{align*}
 \int_{\sqrt{|t|}<|x|<1}
       \left(1+\frac{4|t|^2}{|x|^6}\right)dA(x)
 &=3\pi-\pi|t|-2\pi|t|^2\\
 &\le3\pi.
\end{align*}
The two remaining regular leaves are the punctured coordinate axes, each of area $\pi$. Their closures are the only separatrices through the origin, and their union has finite area. The two-form $-dx\wedge dy$ cannot equal $\alpha\wedge\beta$ with $\beta$ holomorphic, because every coefficient of $\alpha\wedge\beta$ vanishes at the origin.
\end{example}

Taking the product with a polydisc $\Delta^d$ gives nonisolated examples with leaf dimension $k=d+1$ and singular set of dimension $d$. The local volume bounds and finite measure of the two separatrix closures persist.

\subsection{A global consequence in codimension one}

\begin{corollary}\label{cor:malgrange-global}
Let $M$ be a connected compact K\"ahler manifold and let $\F$ be a codimension-one holomorphic foliation whose regular leaf volumes are uniformly bounded. If
\[
 \operatorname{codim}_M\Sing\F\ge3,
\]
then only finitely many regular leaf closures meet $\Sing\F$, and the quotient by finite chains of intersecting leaf closures is a compact Riemann surface. The quotient map is proper and holomorphic with connected fibers.

The same conclusion holds when the singular set has codimension at least two and, at every singular point, a saturated local defining one-form admits a formal unit integrating factor.
\end{corollary}

\begin{proof}
Write $S=\Sing\F$ and $k=n-1$. If $S$ is empty, Theorem~\ref{thm:analytic} applies directly. Otherwise, by Theorem~\ref{thm:malgrange-one}, each $a\in S$ has a neighborhood on which a saturated defining form satisfies
\[
 \alpha=u\,df,\qquad u\ne0.
\]
Thus the critical locus of $f$ is precisely $S$ in that neighborhood. The function $f$ is constant on every irreducible germ of $S$ at $a$, because its differential vanishes there. All these germs pass through $a$, so after shrinking the neighborhood and subtracting $f(a)$, one has $S\subset f^{-1}(0)$.

Choose finitely many such neighborhoods covering the compact set $S$, together with smaller relatively compact neighborhoods still covering $S$. Each hypersurface $f^{-1}(0)$ has only finitely many irreducible components meeting the corresponding smaller neighborhood, by local finiteness of analytic irreducible components.

By Lemma~\ref{lem:closures}, every global regular leaf $L$ has a compact irreducible analytic closure $A_L$, and $A_L\setminus S=L$. Suppose $A_L$ meets $S$ at $s$ in one of the smaller neighborhoods. An irreducible local branch of $A_L$ through $s$ is not contained in $S$, since $\dim S<k$. On its regular part outside $S$, $f$ is constant along the leaf, and hence it is constant on the branch. Its value at $s$ is zero. The branch is therefore a component of $f^{-1}(0)$, since both have dimension $k$.

Distinct global leaves cannot use the same irreducible component of this hypersurface: at a regular point of the foliation on that component, its plaque belongs to a unique leaf, and analytic continuation along the component identifies the same global leaf. The finite collection of hypersurface components obtained from the cover therefore accounts for every global leaf whose closure meets $S$. There are only finitely many such leaves, say $L_1,\ldots,L_N$, and
\[
 \Sigma=\bigcup_{j=1}^N A_{L_j},\qquad
 \HH^{2k}(\Sigma)\le\sum_{j=1}^N\Vol_{2k}(A_{L_j})<\infty,
\]
with the usual Riemannian normalization of Hausdorff measure. The hypotheses of Theorem~\ref{thm:analytic} now hold. Its quotient is normal and one-dimensional, hence a compact Riemann surface. For the second assertion, use the formal-to-holomorphic assertion of Theorem~\ref{thm:malgrange-one} and the same proof.
\end{proof}

\subsection{Malgrange's theorem in arbitrary codimension}

A holomorphic Pfaff system of rank $p$ is here an ordered family of $p$ holomorphic one-form germs $\alpha_1,\ldots,\alpha_p$ which are linearly independent at generic points. Its singular set is
\[
 S(\boldsymbol\alpha)
 =\{x:(\alpha_1\wedge\cdots\wedge\alpha_p)(x)=0\}.
\]
The common kernels define a rank $k=n-p$ distribution outside this set. Frobenius involutivity there is equivalent to
\begin{equation}\label{eq:pfaff-involutive}
 d\alpha_i\wedge\alpha_1\wedge\cdots\wedge\alpha_p=0
 \qquad(1\le i\le p).
\end{equation}
When it holds on the regular locus, this identity of holomorphic forms holds everywhere.

\begin{theorem}[Malgrange, general case]\label{thm:malgrange-general}
Suppose \eqref{eq:pfaff-involutive} holds and
$\operatorname{codim}S(\boldsymbol\alpha)\ge3$. There are holomorphic germs $f_1,\ldots,f_p$ and an invertible holomorphic matrix $G=(g_{ij})$ such that
\[
 \alpha_i=\sum_{j=1}^p g_{ij}\,df_j,
 \qquad\det G(0)\ne0.
\]
The same conclusion holds with singular codimension at least two if such a representation exists formally with an invertible formal matrix.
\end{theorem}

This is \cite[Theorem~3.1, p.~73]{MalgrangeII}, also stated in \cite[Theorem~4.1]{Ramis}. In this representation $S(\boldsymbol\alpha)$ is exactly the critical locus of $F=(f_1,\ldots,f_p)$. The regular leaves are the connected components of the regular parts of its fibers.

For an arbitrary singular foliation, the required presentation is an additional condition. Its conormal sheaf consists of holomorphic one-forms annihilating the regular tangent distribution. A local basis of a locally free conormal sheaf gives such a Pfaff system, with rank-drop set equal to $S$. Local freeness of this sheaf at singular points does not follow from the definition of a singular foliation.

\begin{corollary}\label{cor:malgrange-overlap}
Let $\F$ have leaf dimension $k$ and codimension $p\ge2$ near $a\in M$. Suppose $\dim S<k$ and that $\F$ has a local Pfaff presentation by $p$ one-forms whose rank-drop set is $S$. Then Malgrange's conclusion holds near $a$. This local conclusion requires neither a leaf-volume bound nor the finite-measure hypothesis on $\Sigma$.
\end{corollary}

\begin{proof}
The dimension assumption gives
\[
 c=n-\dim S>n-k=p,
 \qquad\text{hence}\qquad c\ge p+1\ge3.
\]
Involutivity gives \eqref{eq:pfaff-involutive}, and Theorem~\ref{thm:malgrange-general} applies.
\end{proof}

The maximal minors of a $p\times n$ matrix have codimension at most $n-p+1=k+1$ when their zero set is nonempty; Malgrange recalls this determinantal bound at the beginning of \cite[Section~1]{MalgrangeII}. Thus at a singular point with such a presentation the assumptions force
\[
 p+1\le c\le k+1.
\]
In particular, if $p>k$, a singular foliation satisfying $\dim S<k$ cannot have this presentation at its singular points.

\begin{example}[The radial foliation and the conormal condition]\label{ex:radial-malgrange}
Let $n\ge3$ and let
\[
 X=\sum_{j=1}^n z_j\frac{\partial}{\partial z_j}
\]
generate the radial rank-one foliation on a ball $B_\varepsilon\subset\CC^n$. Its singular set is $\{0\}$, so $\dim S<k=1$ and $\operatorname{codim}S=n\ge3$. Nevertheless it has no nonconstant holomorphic first integral near zero. Indeed, writing a holomorphic germ as a sum of homogeneous terms $h=\sum_{m\ge0}h_m$, the equation $X(h)=0$ gives $m h_m=0$ for all $m\ge1$.

Consequently this foliation cannot be defined at the origin by $n-1$ holomorphic one-forms independent everywhere off zero. The determinantal bound also proves this directly: such a presentation would have singular codimension at most two. Thus a foliation defined by a generically decomposable higher-degree form need not admit the Pfaff presentation required in Theorem~\ref{thm:malgrange-general}.

Every regular leaf in the ball is a punctured complex line disc of area $\pi\varepsilon^2$. All leaf closures meet the origin, so the full separatrix set is the entire ball and has infinite $\HH^2$-measure. This example satisfies the uniform volume and dimension conditions but fails the finite-measure condition. Its finite-chain quotient is a point.
\end{example}

\subsection{The local algebra of first integrals}

Malgrange's second paper also gives a factorization result closely related to the sheaf of functions on a quotient. Let $F=(f_1,\ldots,f_p)$ be a holomorphic map germ with $F(0)=0$, generic rank $p$, and critical locus of codimension at least two. Then
\begin{equation}\label{eq:malgrange-composite}
 dg\wedge df_1\wedge\cdots\wedge df_p=0
 \quad\Longrightarrow\quad
 g=\varphi(f_1,\ldots,f_p)
\end{equation}
for a unique $\varphi\in\CC\{t_1,\ldots,t_p\}$; see \cite[Theorem~2.1.1, p.~70]{MalgrangeII}. The division lemma first writes $dg$ as a holomorphic linear combination of the $df_i$, and the theorem integrates this relation to a convergent composite function.

Under the conclusion of Theorem~\ref{thm:malgrange-general}, let $\mathcal I_{\F,0}$ be the ring of holomorphic first-integral germs together with constants. Equation~\eqref{eq:malgrange-composite} gives
\[
 \mathcal I_{\F,0}
 =\{\varphi\circ F:\varphi\in\CC\{t_1,\ldots,t_p\}\}
 \simeq\CC\{t_1,\ldots,t_p\}.
\]
Thus the local algebra of first integrals at one point is a regular analytic algebra. This does not imply that the global leaf quotient is smooth: functions defined near one point need not continue consistently around the entire leaf or around a union of intersecting closures.

For example, at a regular point, a transverse coordinate $z$ is a local first integral. If continuation around a compact leaf acts by $z\mapsto\zeta z$, where $\zeta$ has order $m$, only invariant germs descend around the whole leaf, and
\[
 \CC\{z\}^{\mu_m}=\CC\{z^m\}.
\]
In higher transverse dimension the invariant ring $\CC\{z_1,\ldots,z_p\}^{H_L}$ may be singular.

The relative de Rham complex has degree $j$ term
\[
 \Omega_F^j=\Omega^j\Big/\sum_{i=1}^p df_i\wedge\Omega^{j-1},
\]
with differential induced by exterior differentiation. If $c$ is the codimension of the critical locus, Malgrange proves
\[
 H^0(\Omega_F^\bullet)\simeq\CC\{t_1,\ldots,t_p\}
 \quad(c\ge2),\qquad
 H^j(\Omega_F^\bullet)=0\quad(1\le j\le c-2).
\]
See \cite[Theorems~2.1.1 and~2.2.1]{MalgrangeII}. The degree-zero statement identifies first integrals as functions of $F$. For the global quotient, the identity $\mathcal O_Q=q_*\mathcal O_M$ concerns functions on neighborhoods saturated by entire equivalence classes.

\subsection{Local first integrals and the full separatrix set}

Consider the map germ
\[
 F:\CC_s\times\CC^3_z\longrightarrow\CC^2,
 \qquad F(s,z)=\bigl(s,z_1^2+z_2^2+z_3^2\bigr).
\]
The Pfaff system $ds,d(z_1^2+z_2^2+z_3^2)$ defines a foliation of dimension $k=2$, with
\[
 S=\{z_1=z_2=z_3=0\},\qquad
 \dim S=1<2,
 \qquad\operatorname{codim}S=3.
\]
For each $a$, the punctured quadratic cone in $\{s=a\}$ is a regular leaf near its vertex, and its closure meets $S$ at $(a,0)$. Their union is
\[
 \Sigma=\{z_1^2+z_2^2+z_3^2=0\},
\]
a three-dimensional analytic set in the four-dimensional source. Hence it has infinite $\HH^4$-measure in every neighborhood of zero. The individual fibers have uniformly bounded local four-volume on smaller fixed neighborhoods: they are degree-two algebraic surfaces in the hyperplanes $s=a$, and the usual projective degree bound, together with comparison of Euclidean and Fubini--Study metrics on a compact chart, gives a uniform bound. Malgrange's local integrability and the dimension condition therefore do not imply the finite-measure assumption on the full separatrix set. The singular values form the curve $\{(a,0)\}$.

\subsection{First integrals on saturated neighborhoods}

For a point $b\in Q$ in Theorem~\ref{thm:analytic}, choose an analytic neighborhood $V$ admitting a holomorphic embedding $\iota:V\hookrightarrow\CC^N$. The components of
\[
 \iota\circ q:q^{-1}(V)\longrightarrow\CC^N
\]
are holomorphic functions constant on the regular leaves and separating the complete equivalence classes in $q^{-1}(V)$. They give a system of first integrals on a saturated neighborhood of the entire fiber $q^{-1}(b)$. At smooth points of $Q$ one can take $N=p$; at singular points an embedding generally requires more coordinates and equations among them. Every global holomorphic function on $M$ is constant, since $M$ is compact and connected.

\section{Closed equivalence relations}

Under (i) and (ii), Proposition~\ref{prop:weaker} gives a Hausdorff quotient when only finitely many classes meet $S$ and each is closed. The analytic construction in Theorem~\ref{thm:analytic} uses the stronger conclusion that $\Sigma$ is analytic of pure dimension $k$.

The smallest \emph{closed} equivalence relation containing all pairs of points on the same leaf closure always exists: it is the intersection of all closed equivalence relations containing those pairs, a nonempty family because it contains $M\times M$. Its quotient is compact Hausdorff. This closed relation can be strictly larger than the finite-chain relation and requires no volume hypothesis. Under the hypotheses of Theorem~\ref{thm:main}, it equals $\sim$.

\subsection*{Funding}
This work was supported by Proyecto PAPIIT, Direcci\'on General de Asuntos del Personal Acad\'emico, Universidad Nacional Aut\'onoma de M\'exico [grant number IN103324].

\subsection*{Acknowledgments}

The authors used Claude (Anthropic) and ChatGPT (OpenAI) for proofreading
and editorial assistance.

\bigskip
\noindent\textsc{Alberto Verjovsky}\\
IMATE Cuernavaca, UNAM, Av.~Universidad,\\
62210 Cuernavaca, Morelos, Mexico\\
\textit{Email:} \href{mailto:alberto@matcuer.unam.mx}{\texttt{alberto@matcuer.unam.mx}}

\medskip
\noindent\textsc{Carlos Eduardo Mart\'inez}\\
Unidad de Posgrado, UNAM, Cto.~de los Posgrados,\\
C.U., 04510 Mexico City, CDMX, Mexico\\
\textit{Email:} \href{mailto:cmartineza@ciencias.unam.mx}{\texttt{cmartineza@ciencias.unam.mx}}

\end{document}